\documentclass[10pt,reqno]{amsart}
\usepackage{amsmath,amsopn,amssymb,amsthm}
\usepackage[cp1251]{inputenc}
\usepackage[english]{babel}
\usepackage[pagebackref,breaklinks=true,colorlinks=true,linkcolor=blue,citecolor=blue,urlcolor=blue]{hyperref}
\usepackage{graphicx}%
\usepackage{color}

\newtheorem{theorem}{Theorem}[section]

\newtheorem{lemma}[theorem]{Lemma}

\newtheorem{remark}[theorem]{Remark}

\renewenvironment{proof}[1][Proof]{\noindent\textbf{#1.} }{\ \rule{0.5em}{0.5em}}

\begin{document}
\title[Isoparametric hypersurfaces induced by navigation in Lorentz Finsler geometry]{Isoparametric hypersurfaces induced by navigation in Lorentz Finsler geometry}
\author{Ming Xu, Ju Tan and Na Xu}

\address{Ming Xu \newline
School of Mathematical Sciences,
Capital Normal University,
Beijing 100048,
P.R. China}
\email{mgmgmgxu@163.com}

\address{Ju Tan, corresponding author \newline
School of Mathematics and Physics,
Anhui University of Technology,
Maanshan 243032,
P.R. China}
\email{tanju2007@163.com}

\address{Na Xu\newline
School of Mathematics and Physics,
Anhui University of Technology,
Maanshan 243032,
P.R. China}
\email{xuna406@163.com}
\date{}

\begin{abstract}
Using a navigation process with the datum $(F,V)$, in which $F$ is a Finsler metric and the smooth tangent vector field $V$ satisfies $F(-V(x))>1$ everywhere, a Lorentz Finsler metric $\tilde{F}$
can be induced.
Isoparametric functions and isoparametric hypersurfaces with or without involving a smooth measure can be defined for $\tilde{F}$.
When the vector field $V$ in the navigation datum is homothetic,
we prove the local correspondences between isoparametric functions and isoparametric hypersurfaces before and after this navigation process. Using these correspondences, we provide some examples
of isoparametric functions and isoparametric hypersurfaces on
a Funk space of Lorentz Randers type.

\textbf{Mathematics Subject Classification (2010)}: 52B40, 53C42,
53C60

\textbf{Key words}: Finsler metric, homothetic vector field, isoparametric function, isoparametric hypersurface, Lorentz Finsler metric, Zermelo navigation
\end{abstract}
\maketitle
\section{Introduction}

A smooth function $f$ on a Riemannian manifold $(M,g)$ is called {\it isoparametric} if the pointwise norm of its gradient field $|\nabla f|$ and its Laplacian $\Delta f$ only depends on the values of $f$. The level set of a regular or critical value of $f$ is called an {\it isoparametric hypersurface} or a {\it focal subset} respectively \cite{CR2015}.
These notions have been extensively studied and have many applications in geometry and topology \cite{GT2012}\cite{QT2014}\cite{TY2017}. They were generalized to pseudo-Riemannian geometry and to high codimension during the 1980's \cite{Ha1984}\cite{Te1985}.
The classification of isoparametric hypersurfaces in unit spheres was
one of most celebrated geometric problem \cite{Ya1982}. This problem has a history of eighty years \cite{Ca1939}, and was completely solved by
Q.S. Chi in 2020 \cite{Ch2020}. See \cite{Xu2021} for more references.

In recent years, isoparametric theory has also been studied in Finsler geometry. Using the nonlinear analogs of gradient vector field and Laplacian operators, Q. He et al defined two types of isoparametric functions on a Finsler manifold $(M,F)$, i.e., the isoparametric functions involving a smooth measure $\mathrm{d}\mu$ \cite{HYS2016} (usually $\mathrm{d}\mu$ is taken to be the B.H. measure $\mathrm{d}\mu^F_{\mathrm{Bh}}$ of $F$),
and those without involving any smooth measure
\cite{HYR2021}\cite{Xu2018}. In this paper, they are specified as {\it $\mathrm{d}\mu$-isoparametric functions} and {\it isoparametric functions} respectively.  By a systematical approach in Finsler submanifold geometry, Q. He and her coworkers proved Cartan type formulae\cite{HYR2021}\cite{HYS2016} and
classified $\mathrm{d}\mu^F_{\mathrm{BH}}$-isoparametric hypersurfaces in  Minkowski spaces \cite{HYS2016}, Funk spaces \cite{HYS2017}, and Randers spheres with $K\equiv1$ \cite{HDY2020}.
Notice that these Finsler space forms have constant S-curvature \cite{SS2016}\cite{XMYZ2020}, so the isoparametric property coincides with the $\mathrm{d}\mu^F_{\mathrm{BH}}$-isoparametric property
\cite{Xu2018}.

From another view point, the classifications for isoparametric hypersurfaces in Funk spaces and Randers space forms  can be explained by the correspondence between isoparametric hypersurfaces before and after a {\it homothetic navigation} \cite{XMYZ2020}.  Let $(F,V)$ be the datum in which $V$ is a homothetic field for the Finsler metric $F$, and $\tilde{F}$ the Finsler metric induced by the corresponding navigation process where $F(-V(x))<1$.
Locally any hypersurface is $\mathrm{d}\mu^{\tilde{F}}_{\mathrm{BH}}$-isoparametric for the metric $\tilde{F}$ iff it is $\mathrm{d}\mu^F_{\mathrm{BH}}$-isoparametric for the metric $F$ (see Theorem 1.5 in \cite{XMYZ2020}). The correspondence between isoparametric functions is a little complicated,
but we may simplify it by only considering those isoparametric functions {\it normalized} at some $x_0\in M$ (see Theorem 7.5 in \cite{XMYZ2020}). See \cite{Fo2004}\cite{FM2018}\cite{HM2011}\cite{MH2007}\cite{XMYZ2020}
for more geometric correspondences for homothetic navigation in Finsler geometry.

This correspondence between  isoparametric hypersurfaces totally ignored the domain where $F(-V(x))>1$. In this domain, the navigation process with the datum $(F,V)$ induces
a (conic) Lorentz Finsler metric $\tilde{F}:\mathcal{A}\rightarrow\mathbb{R}_{>0}$, in which $\mathcal{A}=\coprod\mathcal{A}_x$ for all $x\in M$ with $F(-V(x))>1$ and $\mathcal{A}_x=\mathcal{A}\cap T_xM$ consists of all $\tilde{y}=y+F(x,y)V(x)$ satisfying
$\langle y,V(x)\rangle_y^F<-F(x,y)$ (see Lemma \ref{lemma-1}). Moreover, $\tilde{F}$ satisfies Condition (C), i.e., each nonempty $\mathcal{A}_x$ is convex in $T_xM$.
Condition (C) is
useful for us to uniquely define the gradient field $\nabla^{\tilde{F}}f$
of some regular real function $f$. This $\tilde{F}$ was found and studied by E. Caponio et al \cite{CJS2017}. The special case, {\it wind Riemannian} $\tilde{F}$ induced by Riemannian $F$, was studied by M.A. Javaloyes and M. S\'{a}nchez \cite{JS2017-1}\cite{JS2017-2}. The geodesic and flag curvature correspondences for a homothetic navigation in pseudo-Finsler geometry was explored by M.A. Javaloyes and H. Vit\'{o}rio \cite{JV2018}.


The $\mathrm{d}\mu$-isoparametric (and isoparametric)
functions can be generalized to Lorentz Finsler geometry,
using similar
gradient field $\nabla^{\tilde{F}} \tilde{f}$ and Laplacian $\Delta^{\tilde{F}} \tilde{f}$ (and $\Delta^{\tilde{F}}_{\mathrm{d} \mu}\tilde{f}$ respectively). See Section 3.2 for the precise
definitions.
The crucial issue here is that the gradient field $\nabla^{\tilde{F}} \tilde{f}$ for a general real regular function $f$ may not exist, and moreover it may not be unique. But for the situation we discuss below $\nabla^{\tilde{F}} \tilde{f}$ can be uniquely determined.

The main purpose of this paper is to prove the following main theorem, which generalizes
the correspondence between isoparametric hypersurfaces before and after a homothetic navigation with Finsler $F$ and Lorentz Finsler $\tilde{F}$, and provides some examples.

\begin{theorem}\label{main-thm}
Let $V$ be a homothetic vector field with dilation $c$
on the Finsler manifold $(M,F)$,
satisfying $F(-V(x))>1$ everywhere, and $\tilde{F}$ the conic Lorentz Finsler metric induced by the navigation process with the datum $(F,V)$. Denote $\Psi_t$ the flow of homothetic translations generated by $V$, $\mathrm{d}\mu^F_{\mathrm{BH}}$ the B.H. measure of $F$, and $\alpha_c(t)=t$ when $c=0$ or $\alpha_c(t)=\tfrac{e^{2ct}-1}{2c}$
otherwise.
Suppose $f$ is a normalized $\mathrm{d}\mu^F_{\mathrm{BH}}$-isoparametric (or isoparametric) function for $F$ around $x_0$, satisfying
$\langle \nabla^F f(x),V(x)\rangle^F_{\nabla^F f(x)}<-1$, then the function $\tilde{f}$ determined by
$\tilde{f}^{-1}(t)=\Psi_t(f^{-1}(\alpha_c(t)))$ around $x_0$ is a normalized $\mathrm{d}\mu^F_{\mathrm{BH}}$-isoparametric (or isoparametric) function for $\tilde{F}$. In particular,
$f^{-1}(0)=\tilde{f}^{-1}(0)$ is a $\mathrm{d}\mu^F_{\mathrm{BH}}$-isoparametric (or isoparametric respectively) hypersurface around $x_0$ for both $F$ and $\tilde{F}$.
\end{theorem}

To prove Theorem \ref{main-thm}, we use the geodesic correspondence for homothetic navigation to prove $\tilde{f}$ is a normalized transnormal function for $\tilde{F}$ (see Lemma \ref{lemma-12}).
A comparison between $\Delta^F_{\mathrm{d} \mu^F_{\mathrm{BH}}} f$ (see Lemma \ref{lemma-7}) and $\Delta^{\tilde{F}}_{\mathrm{d} \mu^F_{\mathrm{BH}}} \tilde{f}$ indicates $\tilde{f}$ is $\mathrm{d}\mu^F_{\mathrm{BH}}$-isoparametric iff $f$ is $\mathrm{d}\mu^F_{\mathrm{BH}}$-isoparametric.

To find the relation between $\Delta^Ff$ and $\Delta^{\tilde{F}}\tilde{f}$, and prove $\tilde{f}$ is isoparametric for $\tilde{F}$ when $f$ is for $F$, we need a few more ingredients. For the Finsler metric $F$ and the transnormal function $f$,
$\Delta^F f$ only differs from $\Delta^F_{\mathrm{d} \mu^F_\mathrm{BH}} f$ by the S-curvature $S^F_{\mathrm{d} \mu^F_{\mathrm{BH}}}(x,\nabla^F f(x))$ (see Lemma 14.1.2 in \cite{Sh2001} or Lemma \ref{lemma-5} below).
For the Lorentz Finsler metric $\tilde{F}$ and the transmoral function $\tilde{f}$, we can prove a similar equality (see Lemma \ref{lemma-6}). The S-curvatures for $F$ and $\tilde{F}$ are related by Lemma \ref{lemma-8}, which is an analog of Theorem 1.3 in \cite{XMYZ2020}. Then the comparison between
$\Delta^F f$ and $\Delta^{\tilde{F}}\tilde{f}$ can be easily summarized from above observations (see Lemma \ref{lemma-9}).


\begin{remark}\label{remark-0}
We believe that almost the same argument can also prove the backward correspondence, i.e., when $\tilde{f}$ in Theorem \ref{main-thm} is a $\mathrm{d}\mu^F_{\mathrm{BH}}$-isoparametric (or isoparametric) function for $\tilde{F}$ normalized at $x_0$, then $f$ is a $\mathrm{d}\mu^F_{\mathrm{BH}}$-isoparametric (or isoparametric for $F$ respectively) function normalized at $x_0$. Isoparametric functions, isoparametric hypersurfaces, and their correspondence for homothetic navigation can be further generalized to pseudo-Finsler geometry with arbitrary signature type, as long as the existence and choice for the gradient field can be settled.
\end{remark}

Theorem \ref{main-thm} can provide many examples of isoparametric
functions and isoparametric hypersurfaces in Lorentz Finsler geometry. Take a Lorentz Funk metric $\tilde{F}$ for example. When $\tilde{F}$ is a Lorentz Randers metric defined on the complement $M$ of the closed unit ball in
 $\mathbb{R}^n$, it is induced by a homothetic navigation with the datum $(F,V)$, in which $F$ is the standard Euclidean metric and $V(x)=-x$. By \cite{Se1926} and Theorem \ref{main-thm}, hyperplanes, spheres and spherical cylinders provide isoparametric hypersurfaces for $\tilde{F}$ around any $x_0$ with $|x_0|>1$. See Section 4.2 for the precise description of these examples and some explicit isoparametric functions. When $\tilde{F}$ is not Randers, it is induced by the homothetic navigation with the datum $(F,V)$, in which $F$ is a Minkowski metric and $V(x)=-x$ is the same. Then we may use Theorem \ref{main-thm} and \cite{HYS2016} to similarly find isoparametric hypersurfaces for $\tilde{F}$.

This paper is organized as following. In Section 2, we recall some preliminary knowledge on Finsler metric, Lorentz Finsler metric, navigation process and geometric correspondences for homothetic navigation. In Section 3,
we recall the definitions of isoparametric functions and isoparametric hypersurfaces for Finsler metrics (with or without a smooth measure involved), and generalize these notions to Lorentz Finsler metrics. In Section 4, we prove  Theorem \ref{main-thm} and provide some examples
of isoparametric functions and isoparametric hypersurfaces in Funk space of Lorentz Randers type.

\section{Preliminaries}

\subsection{Finsler metric and Lorentz Finsler metric}


A Finsler metric on $M$ is a continuous function $F:TM\rightarrow\mathbb{R}$ satisfying:
\begin{enumerate}
\item the positiveness and smoothness, i.e., $F$ is positive and smooth when restricted to $TM\backslash0$;
\item the positive 1-homogeneity, i.e., $F(x,\lambda y)=\lambda F(x,y)$
for any $x\in M$, $y\in T_xM\backslash\{0\}$ and $\lambda\geq0$;
\item the strong convexity, i.e., for any $y\in T_xM\backslash\{0\}$, the {\it fundamental tensor}
$$\langle u,v\rangle_y^F=\frac12\frac{\partial^2}{\partial s\partial t}|_{s=t=0}
F^2(y+su+tv)$$
is an inner product on $T_xM$.
\end{enumerate}
Using standard local coordinates $(x^i,y^i)$, i.e., $x=(x^i)\in M$ and $y=y^i\partial_{x^i}\in T_xM$, the fundamental tensor can also be presented as $\langle u,v\rangle^F_{y}=g_{ij}u^iv^j$, where
$u=u^i\partial_{x^i}$, $v=v^j\partial_{x^j}$ and  $(g_{ij})=(\frac12[F^2]_{y^iy^j})$ is the Hessian metric \cite{BCS2000}.

More generally,
a ({\it conic}) {\it pseudo-Finsler metric} on $M^n$ is
defined on a conic open subset $\mathcal{A}$ in $TM\backslash0$. The {\it conic} property of $\mathcal{A}\subset TM\backslash0$ is referred to $\mathcal{A}_x=\mathbb{R}_{>0}\mathcal{A}_x$
for each $\mathcal{A}_x=\mathcal{A}\cap T_xM$.
We call the function
$F:\mathcal{A}\rightarrow\mathbb{R}_{>0}$ a {\it pseudo-Finsler metric} on $M$ if the positiveness, smoothness and positive 1-homogeneity are satisfied, and the fundamental tensor $\langle\cdot,\cdot\rangle_y^F$ for any $y\in\mathcal{A}$ is non-degenerate \cite{JV2018}. See \cite{Ko2011} for
the original of pseudo-Finsler metrics in physics.

In this paper, we will study a pseudo-Finsler metric $F:\mathcal{A}\rightarrow\mathbb{R}_{>0}$ satisfying the following properties:
\begin{enumerate}
\item the signature type of $\langle\cdot,\cdot\rangle_y^F$ is $(1,n-1)$ everywhere;
\item it satisfies the following {\it Condition (C)}:  for every $x\in M$, $\mathcal{A}_x=\mathcal{A}\cap T_xM$ is convex.
\end{enumerate}
We will simply call this $F$ a {\it Lorentz Finsler metric satisfying Condition (C)}.


Many geometric notions in Finsler geometry \cite{BCS2000}
can be naturally generalized to pseudo-Finsler geometry. 

The {\it geodesic spray} $G$ is a smooth tangent vector field on $\mathcal{A}$. With respect to standard
local coordinates $(x^i,y^i)$, the geodesic spray has the presentation $$G=y^i\partial_{x^i}-2G^i\partial_{y^i}, \mbox{ in which } G^i=\frac14g^{il}([F^2]_{x^ky^l}y^k-[F^2]_{x^l}).$$

The smooth curve $c(t)$ is called a {\it geodesic} ({\it with positive constant speed}) if its lifting $(c(t),\dot{c}(t))$ is an integral curve of
the geodesic spray $G$. Using standard local coordinates, the geodesic
$c(t)=(c^i(t))$ is determined by the ODE system
 $$\ddot{c}^i(t)+2G^i(c(t),\dot{c}(t))=0,\quad\forall i.$$

Let $\mathrm{d}\mu$ be a smooth measure with local coordinates presentation
$\mathrm{d}\mu=\sigma(x)\mathrm{d}x^1\cdots \mathrm{d}x^n$, in which $\sigma(x)$ is a positive smooth function. Then
the S-curvature $S^F_{\mathrm{d} \mu}:\mathcal{A}\rightarrow \mathbb{R}$ for $F$ and $\mathrm{d}\mu$
is the derivative of the distortion function
$$\tau^F_{\mathrm{d} \mu}(x,y)=\ln\left(\frac{\sqrt{|\det(g_{ij}(x,y))|}}{\sigma(x)}\right)$$ in the direction of the geodesic spray $G$.

\begin{remark}
Notice that in Finsler geometry, there is a canonical smooth measure called the B.H. measure, i.e., $\mathrm{d}\mu^F_{\mathrm{BH}}=\sigma(x)
\mathrm{d}x^1\cdots\mathrm{d}x^n$ with
$$\sigma(x)=\frac{\mathrm{Vol}(\{(y^i)| (\sum_{i=1}^n (y^i)^2)^{1/2}\leq1\})}{
\mathrm{Vol}(\{(y^i)|F(x,y^i\partial_{x^i})\leq1\})},$$
where $\mathrm{Vol}(\cdot)$ is the standard measure in the Euclidean space $\mathbb{R}^n$. Generally speaking, B.H. measure is not well defined in pseudo-Finsler geometry.
\end{remark}
\subsection{Finsler and Lorentz Finsler metrics induced by navigation}

Consider a Finsler metric $F$ on $M$ and a smooth tangent field $V$.
Then the Zermelo navigation process formally defines a new metric $\tilde{F}$ by the equality
\begin{equation}\label{001}
F(x,y)=\tilde{F}(x,y+F(x,y)V(x)),\quad\forall x\in M,y\in
T_xM\backslash\{0\}.
\end{equation}
Geometrically, we parallelly shift the indicatrix $S_x^FM=\{y\in T_xM\mbox{ with }F(x,y)=1\}$ of $F$
by the vector $V(x)$ in $T_xM$, then we get another indicatrix which uniquely determines the metric $\tilde{F}$.

To be precise, the new metric $\tilde{F}$ is still a Finsler metric where $F(-V(x))<1$ is satisfied.
For example, the navigation presentation for the Finsler metric $\tilde{F}$ of Randers type is
\begin{equation*}
\tilde{F}(x,y)=
\frac{\sqrt{(1-\langle V(x),V(x)\rangle)\langle y,y\rangle+\langle y,V(x)\rangle^2}
-\langle V(x),y\rangle}{1-\langle V(x),V(x)\rangle}.
\end{equation*}
Here $\langle\cdot,\cdot\rangle=F(\cdot)^2$ is a Riemannian metric and $V$ is
a tangent field satisfying $\langle V(x),V(x)\rangle<1$ everywhere.

The relation between the fundamental tensors $\langle\cdot,\cdot\rangle_y^F$ and $\langle\cdot,\cdot\rangle^{\tilde{F}}_{\tilde{y}}$ with $F(x,y)=\tilde{F}(x,\tilde{y})=1$ can be revealed by taking the mixed second derivative of (\ref{001}) with respect to the smooth parameters $s$ and $t$ in $y=y(s,t)\in
S^F_xM$ with $\frac{\partial}{\partial s}y(0,0)=u$ and $\frac{\partial}{\partial t}y(0,0)=v$ in $T_y(S^F_xM)=T_{\tilde{y}}
(S^{\tilde{F}}_xM)$, and evaluating it at $s=t=0$, i.e.,
\begin{equation}\label{003}
\langle u,v\rangle^F_y=\frac{1}{1-\langle\tilde{y},V(x)
\rangle^{\tilde{F}}_{\tilde{y}}}
\langle u,v\rangle^{\tilde{F}}_{\tilde{y}}.
\end{equation}
Using the inverse navigation process, $F$ can be induced by the datum $(\tilde{F},-V)$, so we also have
\begin{equation}\label{004}
\langle u,v\rangle^{\tilde{F}}_{\tilde{y}}=
\frac{1}{1+\langle y,V(x)\rangle^F_{y}}\langle u,v\rangle^F_y.
\end{equation}

Notice that above calculation and formulae for fundamental tensors are also valid when $\tilde{F}$ is conic Finsler or conic Lorentz Finsler.

If we have $F(-V(x))\equiv 1$, the induced $\tilde{F}$ is a conic Finsler metric metric. For example, the navigation presentation for the Kropina metric $\tilde{F}$
is $$F(x,y)=\frac{\langle y,y\rangle}{2\langle y,V(x)\rangle},\quad\forall x\in M,y\in T_xM\mbox{ with }\langle y,V(x)\rangle>0,$$
where $\langle \cdot,\cdot\rangle=F(\cdot)^2$ is a Riemannian metric and the tangent field $V$ satisfies
$F(V(x))\equiv1$ everywhere.

The case that $F(-V(x))>1$ is the main concern in this paper. Figure 2 in \cite{JV2018} visualizes the two conic metrics induced by
the navigation process. In later discussion, we just ignore the conic Finsler one and concentrate in the Lorentz Finsler one. Rigorous description for this Lorentz Finsler metric $\tilde{F}$ satisfying Condition (C), induced by the navigation process, is given by the following lemma.

\begin{lemma}\label{lemma-1}
Let $F$ be a Finsler metric and $V$ a smooth tangent field satisfying
$F(-V(x))>1$ everywhere on $M$. Then (\ref{001}), i.e., $F(x,y)=\tilde{F}(x,y+F(x,y)V(x))$ for $y\in T_xM$ satisfying $\langle y,V(x)\rangle^F_y<-F(x,y)$, defines a Lorentz
Finsler metric $\tilde{F}:\mathcal{A}\rightarrow\mathbb{R}_{>0}$,
such that each $\mathcal{A}_x=\mathcal{A}\cap T_xM$ is the nonempty conic convex open subset in $T_xM$ consisting
of all $\tilde{y}=y+F(x,y)V(x)\in T_xM$ with $\langle y,V(x)\rangle^F_y<-F(x,y)$.
\end{lemma}

Lemma \ref{lemma-1} is a consequence of Proposition 2.2 in
\cite{JV2018}. Here we present a self contained proof for it.

\begin{proof}
It is obvious that each $\mathcal{A}_x$ is conic. We will
prove
\begin{equation}\mathcal{A}_x=\{t(y'+V(x))|\forall t>0, y'\in T_xM \mbox{ with }F(x,y')<1\}.\label{006}
\end{equation}

Consider any ray contained in the right side of (\ref{006}), i.e., $\mathbb{R}_{>0}(y'+V(x))$ with $F(x,y')<1$. Denote
$y(t)=y'-t(V(x)+y')$ and $\varphi(t)=F(y(t))$. The function $\varphi(t)$ is continuous and non-negative, and it is smooth when $y(t)\neq0$. Direct calculation shows
\begin{eqnarray}
\frac{\mathrm{d} }{\mathrm{d} t}\varphi(t)=\frac{\langle y(t),-V(x)-y'\rangle^F_{y(t)}}{F(y(t))}\label{005}
\end{eqnarray}
and by Cauchy inequality,
\begin{eqnarray}\label{030}
\frac{\mathrm{d} ^2}{\mathrm{d} t^2}\varphi(t)=\frac{\langle -V(x)-y',-V(x)-y'\rangle_{y(t)}^F
\langle y(t),y(t)\rangle_{y(t)}^{F}
-(\langle -V(x)-y',y(t)\rangle_{y(t)}^F)^2}{F(y(t))^3}\geq0,\label{029}
\end{eqnarray}
so $\varphi(t)$ is convex.
Since $\varphi(0)=F(y')<1$ and $\varphi(1)=F(-V(x))>1$, there exists $y=y(t_0)$ for some $t_0\in(0,1)$ with $\varphi(t_0)=F(y)=1>f(0)$. By the convexity of $f(t)$ and (\ref{005}),
$$\frac{\mathrm{d} }{\mathrm{d} t}\varphi(t_0)=\frac{\langle y,-V(x)-y'\rangle_y^F}{F(y)}=\frac{
\langle y,-V(x)-y\rangle_y^F}{
(1-t_0)F(y)}=\frac{-\langle y,V\rangle_y^F-1}{(1-t_0)F(y)}>0.$$
So $\langle y,V(x)\rangle_y^F<-F(y)=-1$ and $$\mathbb{R}_{>0}\tilde{y}=\mathbb{R}_{>0}(y+V(x))=
\mathbb{R}_{>0}((1-t_0)(y'+V(x)))
=\mathbb{R}_{>0}(y'+V(x))$$ is contained in $\mathcal{A}_x$.

On the other hand, consider any ray $\mathbb{R}_{>0}\tilde{y}$ in
$ \mathcal{A}_x$ where $\tilde{y}=y+V(x)$ with $g_y^F(y,V(x))<-F(x,y)=-1$. Similar calculation as (\ref{005}) shows $\varphi(t)=F(y-t(y+V(x)))$ has positive derivative at $t=0$, so for $y'=y-t'(y+V(x))$ with negative $t'$ sufficiently close to 0, $F(y')<1$
and $\mathbb{R}_{>0}(y+V(x))=\mathbb{R}_{>0}(y'+V(x))$ is contained in the right side of (\ref{006}).

To summarize, above argument proves the presentation (\ref{006}) for $\mathcal{A}_x$. By this presentation, we see $\mathcal{A}_x$ is a cone
generated by the nonempty convex open subset $\{y+V(x),\forall y\in T_xM\mbox{ with }F(y)<1\}$. So $\mathcal{A}_x$ is a nonempty convex conic open subset in $T_xM$. On the other hand, the presentation (\ref{006}) depends continuously on $x\in M$, so
$\mathcal{A}=\coprod_{x\in M}\mathcal{A}_x$ is an open subset in
$TM\backslash0$, i.e., a Lorentz Finsler metric may be defined on it.

Secondly, we proves (\ref{001}) defines a positive smooth function $\tilde{F}$ on $\mathcal{A}$. Denote the open subset $\mathcal{A}'=\coprod_{x\in M}\mathcal{A}'_x$ in $TM\backslash0$ in which $\mathcal{A}'_x$ is the subset of all $y\in T_xM\backslash0$ satisfying $\langle y,V(x)\rangle^F_{y}<-F(x,y)$, and $\phi(x,y)=(x,y+F(x,y)V(x))$ the surjective smooth map from $\mathcal{A}'$ to $\mathcal{A}$. Now we prove $\phi$ is a diffeomorphism. Direct calculation shows its tangent map is isomorphic everywhere. To see it is injective, we consider the possibility of $y_1+F(x,y_1)V(x)=y_2+F(x,y_2)V(x)=v$ for any
$y_1,y_2\in \mathcal{A}'_x$. If $F(x,y_1)=F(x,y_2)$, we
get immediately $y_1=y_2$. Otherwise, we may assume $F(x,y_1)<F(x,y_2)$. By similar calculation as (\ref{005}) and (\ref{030}), $\psi(t)=F(x,v-tV(x))$ is a continuous convex non-negative function
with
\begin{equation*}
\tfrac{\mathrm{d}}{\mathrm{d}t}\psi(F(x,y_1))=
\frac{\langle y_1,-V(x)\rangle^F_{y_1}}{F(x,y_1)}>1\quad\mbox{and}\quad
\tfrac{\mathrm{d}}{\mathrm{d}t}\psi(F(x,y_2))=
\frac{\langle y_2,-V(x)\rangle^F_{y_2}}{F(x,y_2)}>1,
\end{equation*}
and $\psi(t)$ is smooth when $\psi(t)\neq0$.
So the restriction of $\psi(t)$ to the interval $[F(x,y_1),F(x,y_2)]$
is positive and smooth with $\tfrac{\mathrm{d}}{\mathrm{d}t}\psi(t)>1$ everywhere. This is a contradiction to Lagrange's Mean Value Theorem because $\psi(F(x,y_i))=F(x,y_i)$ for $i=1,2$.
To summarize, $\phi$ is a diffeomorphism
from $\mathcal{A}'$ to $\mathcal{A}$, and then (\ref{001}) implies
$\tilde{F}=F\circ\phi^{-1}$
is a positive smooth function on $\mathcal{A}$. The positive 1-homogeneity of $\tilde{F}$ is obvious.

Finally, we prove the fundamental tensor of $\tilde{F}$ has the signature type $(1,n-1)$. At each $\tilde{y}=y+F(x,y)V(x)$ with $\tilde{F}(x,\tilde{y})=F(x,y)=1$ and $\langle y,V(x)\rangle^F_y<-1$, the fundamental tensor $\langle\cdot,\cdot\rangle^{\tilde{F}}_{\tilde{y}}$ is obviously
positive definite in the direction $\mathbb{R}\tilde{y}$.
Apply $\langle y,V(x)\rangle^F_y<-F(x,y)=-1$ to (\ref{004}), we see $\langle\cdot,\cdot\rangle^{\tilde{F}}_{\tilde{y}}$ is negative definite when restricted to $T_{\tilde{y}}S^{\tilde{F}}_x M=T_yS^F_xM$, i.e., the $\langle\cdot,\cdot\rangle^{\tilde{F}}_{\tilde{y}}$-orthogonal complement of $\mathbb{R}\tilde{y}$. So $\langle\cdot,\cdot\rangle^{\tilde{F}}_{\tilde{y}}$ has the signature type $(1,n-1)$.

To summarize, $\tilde{F}:\mathcal{A}\rightarrow\mathbb{R}_{>0}$ is a Lorentz Finsler metric, and it satisfies Condition (C).
\end{proof}

\subsection{Homothetic navigation and geometric correspondences}

A smooth tangent vector field $V$ on the Finsler manifold $(M,F)$
is called {\it homothetic} if it generates a one-parameter local subgroup of local homothetic translations $\Psi_t$ satisfying
\begin{equation}\label{007}
(\Psi_t^*F)(x,y)=F(\Psi_t(x),(\Psi_t)_*(y))=e^{-2ct}F(x,y),
\end{equation}
for each $x\in M$, $y\in T_xM$ and $t\in \mathbb{R}$, whenever
$\Psi_t(x)$ is defined. The constant $c$ in (\ref{007}) is called
the {\it dilation} of $V$. We call the homothetic field $V$ a {\it Killing field}
when its dilation $c=0$.
Using the standard local coordinates $(x^i,y^i)$ on $TM$, the homothetic property for $V$ can be equivalently described by $V^{\mathrm{c}} F=-2cF$, in which $V^{\mathrm{c}}=V^i\partial_{x^i}
+y^j\frac{\partial}{\partial x^j}V^i\partial_{y^i}$ is the complete lifting of $V=V^i(x)\partial_{x^i}$. This observation implies the following useful fact (see Lemma 3.1 in \cite{XMYZ2020}).
\begin{lemma}\label{lemma-10}
The restriction of the homothetic field $V$ with dilation $c$ to a unit speed geodesic $\gamma(t)$ on the Finsler manifold $(M,F)$ satisfies
    \begin{equation}\label{008}
    \langle V(\gamma(t)),\dot{\gamma}(t)\rangle^F_{\dot{\gamma}(t)}
    \equiv c_0-2ct,
    \end{equation}
    in which the constant $c_0$ may be taken as $c_0=\langle V(\gamma(0)),\dot{\gamma}(0)\rangle^F_{\dot{\gamma}(0)}$ when $\gamma(t)$ is defined for $t$ around $0$.
\end{lemma}

The equality (\ref{007}) for
$\Psi_t$ generated by $V$ immediately implies
\begin{equation}\label{016}
\langle(\Psi_t)_*u,(\Psi_t)_*v\rangle_{(\Psi_t)_*y}^F=
e^{-4ct}\langle u,v\rangle_y^F, \quad\forall u,v\in T_xM,
\end{equation}
from which we see the tangent map $(\Psi_t)_*:TM\rightarrow TM$ preserves the distortion function $\tau^F_{\mathrm{d}\mu^F_{\mathrm{BH}}}(x,y)$ (defining the S-curvature $S^F_{\mathrm{d}\mu^F_{\mathrm{BH}}}$ for $F$ and its B.H. measure) on $TM\backslash0$. Meanwhile, $(\Psi_t)_*$ preserves the geodesic spray $G$ because $\Psi_t$ maps any constant speed geodesic to a constant speed geodesic, so $(\Psi_t)_*$ preserves the S-curvature $S^F_{\mathrm{d}\mu^F_{\mathrm{BH}}}$, i.e.,
$S^F(x,y)=S^F(\Psi_t(x),(\Psi_t)_*y)$, $\forall (x,y)\in TM\backslash0$ and $t\in\mathbb{R}$ where $\Psi_t$ is defined. This observation is valid not only for those $\Psi_t$ generated by a homothetic vector field but also for all local homothetic translations, which can be summarized as following.
\begin{lemma}\label{lemma-11}
Any local homothetic translation on a Finsler manifold $(M,F)$ preserves the S-curvature $S^F_{\mathrm{d}\mu^F_{\mathrm{BH}}}$.
\end{lemma}

We say $\tilde{F}$ is induced by a
{\it homothetic (or Killing) navigation}
if the tangent vector field $V$ in the navigation
datum $(F,V)$ is a homothetic (or Killing respectively) field.
The geometric correspondences between $F$ and $\tilde{F}$ have been extensively explored in Finsler and pseudo-Finsler geometry \cite{JV2018}\cite{XMYZ2020}. In this paper,
we will use the correspondences for the geodesics and for S-curvatures before and after a homothetic navigation when $F$ is Finsler and $\tilde{F}$ is Lorentz Finsler.

To unify the discussion for Killing navigation and non-Killing homothetic navigation, we
denote $\alpha_0(t)=t$ and  $\alpha_c(t)=\frac{e^{2ct}-1}{2c}$ for $c\in\mathbb{R}\backslash\{0\}$, i.e.,
$\alpha_c(t)$ satisfies $\frac{\mathrm{d} }{\mathrm{d} t}\alpha_c(t)=e^{2ct}$ and $\alpha_c(0)=0$. The variable $t$ here and later can always be assumed to be sufficiently close to 0, because our discussion is local.

\begin{lemma}\label{lemma-2}
Let $\tilde{F}$ be the Lorentz Finsler metric induced by the navigation process with the datum $(F,V)$, in which $V$ is a homothetic field for $F$ with dilation $c$, satisfying
$F(-V(x))>1$ everywhere. Denote $\Psi_t$ the flow of local diffeomorphisms generated by $V$.
Suppose $\gamma(t)$ is a unit speed geodesic on $(M,F)$ satisfying
$c_0=\langle \dot{\gamma}(0),V(\gamma(0))
\rangle^F_{\dot{\gamma}(0)}<-1$. Then  $\tilde{\gamma}(t)
=\Psi_t(\gamma(\alpha_c(t)))$ (for $t$ sufficiently close to 0) is a unit speed geodesic on $(M,\tilde{F})$. Further more,
for any tangent vectors $v_1,v_2\in T_{\gamma(\alpha_c(t))}M$
with $$\langle v_1,\dot{\gamma}(\alpha_c(t))\rangle^F_{\dot{\gamma}(\alpha_c(t))}
=\langle v_2,\dot{\gamma}(\alpha_c(t))\rangle^F_{\dot{\gamma}(\alpha_c(t))}
=0,$$
$(\Psi_t)_*(v_1)$ and $(\Psi_t)_*(v_2)$ in $ T_{\tilde{\gamma}(t)}M$
satisfy
$$\langle (\Psi_t)_*(v_1),\dot{\tilde{\gamma}}(t)\rangle^{\tilde{F}}_{
\dot{\tilde{\gamma}}(t)}
=\langle (\Psi_t)_*(v_2),\dot{\tilde{\gamma}}(t)\rangle^{\tilde{F}}_{
\dot{\tilde{\gamma}}(t)}
=0,$$
and
\begin{equation}\label{009}
\langle(\Psi_t)_*(v_1),
(\Psi_t)_*(v_2),\rangle^{\tilde{F}}_{\dot{\tilde{\gamma}}(t)}
=\frac{1}{c_0+1}e^{-2ct}\langle v_1,v_2\rangle^F_{\dot{\gamma}(\alpha_c(t))}.
\end{equation}
\end{lemma}

\begin{proof}
The first statement in
Lemma \ref{lemma-2} is a special case of Theorem 1.2 in \cite{JV2018}. For the convenience in later discussion, we denote
$y(t)=(\Psi_t)_*(\frac{\mathrm{d} }{\mathrm{d} t}\gamma(\alpha_c(t)))$ and present
\begin{equation}\label{015}
\dot{\tilde{\gamma}}(t)=
y(t)+V(\tilde{\gamma}(t)).
\end{equation}
By the homothetic property of $\Psi_t$,
$$y(t)=(\Psi_t)_*(\frac{\mathrm{d} }{\mathrm{d} t}\gamma(\alpha_c(t)))=
e^{2ct}(\Psi_t)_*(\dot{\gamma}(\alpha_c(t)))$$ is a $F$-unit vector.
Using (\ref{008}) and (\ref{016}), we get
\begin{eqnarray}\label{017}
\langle y(t),
V(\tilde{\gamma}(t))\rangle^F_{y(t)}
&=&\langle (\Psi_t)_*(\frac{\mathrm{d} }{\mathrm{d} t}\gamma(\alpha_c(t))),
(\Psi_t)_*(V({\gamma}(\alpha_c(t)))\rangle^F_{
(\Psi_t)_*(\frac{\mathrm{d} }{\mathrm{d} t}\gamma(\alpha_c(t)))}\nonumber\\
&=&e^{-4ct}\langle\frac{\mathrm{d} }{\mathrm{d} t}\gamma(\alpha_c(t)),
V(\gamma(\alpha_c(t)))
\rangle^F_{\dot{\gamma}(\alpha_c(t))}\nonumber\\
&=&e^{-2ct}\langle
\dot{\gamma}(\alpha_c(t)),V(\gamma(\alpha_c(t)))
\rangle^F_{\dot{\gamma}(
\alpha_c(t))}\nonumber\\
&=&e^{-2ct}(c_0-2c\alpha_c(t))\nonumber\\
&=&e^{-2ct}(c_0+1)-1,
\end{eqnarray}
in which $c_0=\langle\dot{\gamma}(0),
V(\gamma(0))\rangle<-1$.
So $$\langle y(t),
V(\tilde{\gamma}(t))\rangle^F_{
y(t)}<
-F(\tilde{\gamma}(t),y(t))=-1,
$$
and
$\dot{\tilde{\gamma}}(t)=y(t)+V(\tilde{\gamma}(t))$ is a $\tilde{F}$-unit vector in $\mathcal{A}_{\tilde{\gamma}(t)}$.

The other statements in Lemma \ref{lemma-2} then follows (\ref{004}) and (\ref{017}) by almost the same argument as for Lemma 5.2 in \cite{XMYZ2020}.
\end{proof}
\begin{lemma} \label{lemma-8}
Let $\tilde{F}$ be the Lorentz Finsler metric on $M^n$ induced by the homothetic navigation process with the datum $(F,V)$, in which the homothetic tangent vector field $V$ has the dilation constant $c$ and satisfies
$F(-V(x))>1$ everywhere. Then with respect to the B.H. measure
$\mathrm{d}\mu^F_{\mathrm{BH}}$ of $F$, the S-curvatures $S^F_{\mathrm{d}\mu^F_{\mathrm{BH}}}$ for $F$ and $S^{\tilde{F}}_{\mathrm{d}\mu^F_{\mathrm{BH}}}$ for $\tilde{F}$, are related by the following equality,
\begin{equation}\label{010}
S^{\tilde{F}}_{\mathrm{d}\mu^F_{\mathrm{BH}}}(x,\tilde{y})=
S^F_{\mathrm{d}\mu^F_{\mathrm{BH}}}(x,y)+(n+1)cF(x,y),
\end{equation}
for any $y\in T_xM\backslash\{0\}$ satisfying $\langle y,V(x)\rangle_y^F<-F(y)$ and $\tilde{y}=y+F(x,y)V(x)$.
\end{lemma}
\begin{proof} This proof is very similar to that of Theorem 1.3 in \cite{XMYZ2020}. We only need to prove (\ref{010})
when $F(x,y)=\tilde{F}(x,\tilde{y})=1$.
Let $\gamma(t)$ be the unit speed geodesic on $(M,F)$
satisfying $\gamma(0)=x$ and $\dot{\gamma}(0)=y$,
and denote $\tilde{\gamma}(t)=\Psi_t(\gamma(\alpha_c(t)))$ in which $\Psi_t$ is the flow of local diffeomorphisms generated by $V$.

We choose the smooth vector fields
$e_i(t)$, $1\leq i\leq n$, along $\gamma(t)$, such that
$e_1(t)=\dot{\gamma}(t)$ and they are orthonormal  w.r.t. $\langle\cdot,\cdot\rangle^F_{\dot{\gamma}(t)}$ at each $\gamma(t)$. Then we have the induced smooth vector fields along $\tilde{\gamma}(t)$:
\begin{eqnarray*}
\bar{e}_1(t)&=&(\Psi_t)_*(\frac{\mathrm{d} }{\mathrm{d} t}\gamma(\alpha_c(t)))
=e^{2ct}(\Psi_t)_*(e_1(\alpha_c(t))),\\
\bar{e}_i(t)&=&(\Psi_t)_*(e_i(\alpha_c(t))),
\quad\forall 1<i\leq n;\\
\tilde{e}_1(t)&=&\dot{\tilde{\gamma}}(t)=
\bar{e}_1(t)+V(\tilde{\gamma}(t)),\\
\tilde{e}_i(t)&=&\bar{e}_i(t),\quad\forall 1<i\leq n.
\end{eqnarray*}
Then at each $\tilde{\gamma}(t)$, $\{\bar{e}_i(t),\forall 1\leq i\leq n\}$ and $\{\tilde{e}_i(t),\forall 1\leq i\leq n\}$ are two bases for
$T_{\tilde{\gamma}(t)}M$. Denote
\begin{eqnarray*}
\mathrm{vol}(t)&=&\mathrm{Vol}(\{(y^i)|
F(\gamma(t),y^ie_i(t))\leq1\}),\\
\overline{\mathrm{vol}}(t)&=&\mathrm{Vol}(\{(y^i)|
F(\tilde{\gamma}(t),y^i\bar{e}_i(t))\leq1\}),\\
\widetilde{\mathrm{vol}}(t)&=&\mathrm{Vol}(\{(y^i)|
\tilde{F}(\tilde{\gamma}(t),y^i\tilde{e}_i(t))\leq1\}),
\end{eqnarray*}
in which $\mathrm{Vol}(\cdot)$ is the standard measure in the Euclidean space $\mathbb{R}^n$.

The distortion function $\tau^F_{\mathrm{d}\mu^F_{\mathrm{BH}}}$
can be calculated  at each $(\gamma(t),\dot{\gamma}(t))$ using $\{e_i(t),\forall 1\leq i\leq n\}$, i.e.,
\begin{equation}\label{011}
\tau^F_{\mathrm{d}\mu^F_{\mathrm{BH}}}(\gamma(t),\dot{\gamma}(t))
=\ln\sqrt{\det(\langle e_i(t),e_j(t)\rangle^F_{\dot{\gamma}(t)})}
+\ln\mathrm{vol}(t)+C_0,
\end{equation}
in which $C_0$ is some universal constant depending on $n$.
Similarly the distortion function $\tau^{\tilde{F}}_{\mathrm{d}\mu^F_{\mathrm{BH}}}$
can be presented as
\begin{equation}\label{012}
\tau^{\tilde{F}}_{\mathrm{d}\mu^F_{\mathrm{BH}}}
(\gamma(t),\dot{\gamma}(t))
=\ln\sqrt{|\det(\langle \tilde{e}_i(t),\tilde{e}_j(t)\rangle^{\tilde{F}}_{
\dot{\tilde{\gamma}}(t)})|}+\ln\widetilde{\mathrm{vol}}(t)+C_0.
\end{equation}

Using Lemma \ref{lemma-2}, the relation between the two determinants in (\ref{011}) and (\ref{012}) can be given as
\begin{equation}\label{013}
|\det(\langle\tilde{e}_i(t),\tilde{e}_j(t)
\rangle^{\tilde{F}}_{\dot{\tilde{\gamma}}(t)})|
=\frac{1}{|1+c_0|^{n-1}}e^{-2c(n-1)t}\det(\langle e_i(\alpha_c(t)),e_j(\alpha_c(t))
\rangle_{\dot{\gamma}(t)})^F,
\end{equation}
where $c_0=\langle\dot{\gamma}(0),V(\gamma(0))
\rangle_{\dot{\gamma}(0)}^F<-1$.

By multi-variable integral, (\ref{017}) and the homothetic property of $V$,
\begin{eqnarray}\label{018}
\widetilde{\mathrm{vol}}(t)&=&
\frac{\overline{\mathrm{vol}}(t)}{|1+
\langle V(\tilde{\gamma}(t)),\bar{e}_1(t)
\rangle^F_{\bar{e}_1(t)}|}=
\frac{1}{|1+c_0|}e^{2ct}
\overline{\mathrm{vol}}(t)\nonumber\\
&=&\frac{1}{|1+c_0|}e^{2ct} e^{2c(n-1)t}\mathrm{vol}(\alpha_c(t))=
\frac{1}{|1+c_0|}e^{2cnt}\mathrm{vol}(\alpha_c(t)).
\end{eqnarray}

Summarizing (\ref{011})-(\ref{018}), we get
$$\tau^{\tilde{F}}_{\mathrm{d}\mu^F_{\mathrm{BH}}}(\tilde{\gamma}(t),\dot{\tilde{\gamma}}(t))
=\tau^F_{\mathrm{d}\mu^F_{\mathrm{BH}}}(\gamma(\alpha(t)),\dot{\gamma}(\alpha_c(t)))
+c(n+1)t-\frac{n+1}2\ln |1+c_0|,$$
so
\begin{eqnarray*}
S^{\tilde{F}}_{\mathrm{d}\mu^F_{\mathrm{BH}}}(x,\tilde{y})&=&\frac{\mathrm{d} }{\mathrm{d} t}
\tau^{\tilde{F}}_{\mathrm{d}\mu^F_{\mathrm{BH}}}(\tilde{\gamma}(t),\dot{\tilde{\gamma}}(t))
|_{t=0}\\&=&
\frac{\mathrm{d} }{\mathrm{d} t}\alpha_c(0)\frac{\mathrm{d} }{\mathrm{d} s}
\tau^F_{\mathrm{d}\mu^F_{\mathrm{BH}}}(\gamma(s),\dot{\gamma}(s))|_{s=0}+(n+1)c\\
&=&S^F_{\mathrm{d}\mu^F_{\mathrm{BH}}}(x,y)+(n+1)c,
\end{eqnarray*}
which ends the proof.
\end{proof}
\section{Isoparametric hypersurface}

\subsection{Isoparametric hypersurface in Finsler geometry}
Suppose $f$ is a regular real function on the Finsler manifold $(M,F)$ (so the smooth one-form $\mathrm{d}f$ is non-vanishing everywhere).
Then its gradient field $\nabla^F f$ is a smooth tangent vector field on $M$, non-vanishing everywhere, determined by
$\langle\nabla^F f,u\rangle^F_{\nabla f}=(\mathrm{d}f)(u)=u\cdot f$ for each tangent vector $u$, or equivalently, $\nabla f=\mathcal{L}^{-1}(\mathrm{d}f)$, where
$\mathcal{L}:TM\backslash0\rightarrow T^*M\backslash0$ is the Legendre transformation.
The Laplacian of $f$
with respect to  the smooth measure $\mathrm{d}\mu$ is
\begin{equation}\label{022}
\Delta^F_{\mathrm{d} \mu}f=\mathrm{div}_{\mathrm{d} \mu}
\nabla^F f,
\end{equation}
The divergence in (\ref{022}) is defined by
$\mathrm{div}_{\mathrm{d} \mu}X=L_{X}\mathrm{d}\mu$ for any smooth tangent vector field $X$, in which $L$ is the Lie derivative. In particular, when $\mathrm{d}\mu$ coincides with the measure for the
osculation metric $g^F_{\nabla f}$, the corresponding $\Delta^F_{\mathrm{d} \mu}f$ is simply denoted as $\Delta^F f$.
$\Delta^F f$ is in fact the Laplacian of $f$ with respect to the osculation metric $g^F_{\nabla^F f}$.

We call the regular real function $f$ on the Finsler manifold $(M,F)$ {\it transnormal}, if $F(\nabla^F f)$ only depends on the values of $f$. Further more, we call the transnormal function $f$ {\it $\mathrm{d}\mu$-isoparametric}
if $\Delta^F_{\mathrm{d} \mu} f$ only depends on the values of $f$.
The level sets of a $d\mu$-isoparametric function $f$ is called a {\it $\mathrm{d}\mu$-isoparametric hypersurface}. Using $\Delta^F f$ instead of $\Delta^F_{d\mu}f$ in above notions, we get {\it isoparametric function} and {\it isoparametric hypersurface} instead.
In this paper, our discussion for a $d\mu$-isoparametric or isoparametric function is local where no critical value or {\it focal set} is involved.

Replacing $f$ with $\varphi\circ f$ for some suitable one-variable smooth real function $\varphi$, we may {\it normalize} a transnormal,
$\mathrm{d}\mu$-isoparametric, or isoparametric function at some $x_0\in M$ such that
$F(\nabla f)\equiv1$ and $ f(x_0)=0$ (see Lemma 4.1 in \cite{HYS2016}). Notice that the level sets before and after this normalization are essentially the same.

The relation between $\Delta^F_{\mathrm{d} \mu} f$ and $\Delta^F f$ is revealed by Lemma 14.1.2 in \cite{Sh2001},
which is reformulated as following.

\begin{lemma} \label{lemma-5}
Suppose $f$ is a regular real function on a Finsler manifold $(M,F)$,
such that the integral curves of $\nabla^F f$ are $F$-unit speed geodesics, then we have
\begin{equation}\label{023}
\Delta^F f(x)=\Delta^F_{\mathrm{d} \mu} f(x)+S^F_{\mathrm{d}\mu}(x,\nabla^F f(x)),
\end{equation}
in which the S-curvature $S^F_{\mathrm{d}\mu}$ is for $F$ and the smooth measure $\mathrm{d}\mu$.
\end{lemma}

\begin{remark}\label{remark-2}
Lemma \ref{lemma-5} is enough for our later usage. However, similar argument can prove (\ref{023}) when integral curves of $\nabla^F f$
are geodesics with possibly nonconstant speed. So by Lemma 4.1 in \cite{Xu2018},
any transnormal function $f$ on the Finsler manifold $(M,F)$ satisfies
(\ref{023}) in Lemma \ref{lemma-5}.
\end{remark}
\subsection{Isoparametric hypersurface in a Lorentz Finsler manifold}

Let $F:\mathcal{A}\rightarrow\mathbb{R}_{>0}$
be a Lorentz Finsler metric satisfying Condition (C), i.e., for each $x\in M$, $\mathcal{A}_x=\mathcal{A}\cap T_xM$ is a nonempty conic convex open subset of $T_xM$.

For every $x\in M$, the Legendre transformation $\mathcal{L}_x(y)=\langle\cdot,y\rangle^F_y$ for $y\in\mathcal{A}_x$ is a positively 1-homogeneous smooth map with isomorphic tangent maps everywhere, so $\mathcal{L}_x(\mathcal{A}_x)$ is a nonempty conic open subset in
$T_x^*M$.
\begin{lemma} \label{lemma-3}
The Legendre transformation $\mathcal{L}_x$ is injective.
\end{lemma}
\begin{proof} Assume conversely that $\mathcal{L}_x(y'_0)=\mathcal{L}_x(y'_1)$ for different $y'_0,y'_1\in\mathcal{A}_x$. Then $y_0=\tfrac{y'_0}{F(x,y'_0)}$
and $y_1=\tfrac{y'_1}{F(x,y'_1)}$ in $S^{\tilde{F}}_xM$ are linearly independent and $\mathrm{ker}\mathcal{L}_{y_0}=\mathrm{ker}\mathcal{L}_{y_1}$. Since $F$ satisfies Condition (C),
$\varphi(t)=F(y_t)$ with $y_t=(1-t)y_0+ty_1$
is a smooth function for $t\in[0,1]$. Similar calculation as (\ref{029}) shows
\begin{equation*}
\frac{\mathrm{d}^2}{\mathrm{d}t^2}\varphi(t)
=\frac{\langle y_1-y_0,y_1-y_0\rangle^F_{y_t}
\langle y_t,y_t\rangle_{y_t}^F-
(\langle y_t,y_1-y_0\rangle^F_{y_t})^2}{F(x,y_t)^3}
\end{equation*}
is negative for every $t\in[0,1]$ because
$\langle\cdot,\cdot\rangle_{y_t}^F$ has the signature type $(1,n-1)$ and the two vectors $y_1-y_0$ and $y_t$ are linearly independent. Because $\varphi(0)=\varphi(1)=1$, the concavity of $\varphi(t)$ implies $\frac{\mathrm{d}}{\mathrm{d}t}
\varphi(0)>0$ and $\frac{\mathrm{d}}{\mathrm{d}t}
\varphi(1)<0$. So the following vectors
$$v_0=(y_1-y_0)-\langle y_1-y_0,y_0\rangle^F_{y_0}y_0=
(y_1-y_0)-\frac{\mathrm{d}}{\mathrm{d}t}\varphi(0)y_0
=(-1-\frac{\mathrm{d}}{\mathrm{d}t}\varphi(0))y_0+y_1$$
and
$$v_1=(y_1-y_0)-\langle y_1-y_0,y_1\rangle^F_{y_1}
=(y_1-y_0)-\frac{\mathrm{d}}{\mathrm{d}t}\varphi(1)y_1
=-y_0+(1-\frac{\mathrm{d}}{\mathrm{d}t}\varphi(1))y_1$$
are linearly independent.
However, $\mathrm{ker}\mathcal{L}_x(y_i)\cap\mathrm{span}\{y_0,y_1\}
=\mathbb{R}v_i$. This is a contradiction to $\mathrm{ker}\mathcal{L}_x(y_0)=\mathrm{ker}\mathcal{L}_x(y_0)$.
\end{proof}

Denote $\mathcal{L}(x,y)=\mathcal{L}_x(y)$ for every $x\in M$ and $y\in\mathcal{A}_x$. Then we see from Lemma \ref{lemma-3} that
$\mathcal{L}$ is a diffeomorphism from $\mathcal{A}$ to the conic open subset $\mathcal{L}(\mathcal{A})=\coprod_{x\in M}\mathcal{L}_{x}(\mathcal{A}_x)\subset T^*M\backslash0$.

Now we consider a regular real function $f$ defined around $x_0\in M$. Assume that
the value of $\mathrm{d}f$ at $x_0$ is contained in $\mathcal{L}_{x_0}(\mathcal{A}_{x_0})$, then $\nabla^F f=\mathcal{L}^{-1}(\mathrm{d}f)$ uniquely defines a smooth tangent vector field on some sufficiently small neighborhood $\mathcal{U}$ of $x_0$. Equivalently, $\nabla^F f$ can be determined by $\langle u,\nabla^F f\rangle^F_{\nabla f}=\mathrm{d}f(u)$, $\forall x\in\mathcal{U}$ and $u\in T_xM$. We will call it the {\it gradient field} of $f$.

Using $\nabla^F f$ around $x_0$, the Laplacian $\Delta^F_{\mathrm{d} \mu}f=\mathrm{div}_{\mathrm{d} \mu}\nabla^F f$ for the smooth measure $\mathrm{d}\mu$ can be similarly defined as in Finsler geometry.
When  $\mathrm{d}\mu=\sqrt{|\det(g_{ij}(x,\nabla f))|}\mathrm{d}x^1\cdots \mathrm{d}x^n$, i.e., it is the one induced by the osculation Lorentz metric $g^F_{\nabla f}$, the corresponding Laplacian is simply denoted as
$\Delta^F f$, which coincides with the Laplacian of $f$ for the
osculation Lorentz metric $g^F_{\nabla^F f}$ \cite{Ne1983}.


We call $f$ {\it transnormal} if $F(\nabla^F f)$ only depends on the values of $f$. Further more, we call the transnormal function $f$ {\it $\mathrm{d}\mu$-isoparametric} if $\Delta^F_{\mathrm{d} \mu}f$ only depends on the values of $f$. The level set of a $d\mu$-isoparametric function $f$ is called a {\it $\mathrm{d}\mu$-isoparametric} hypersurface. If we replace $\Delta^F_{d\mu} f$ by $\Delta^F f$
in above notions,  {\it isoparametric function} and {\it isoparametric hypersurface} are defined.

The relation between $\Delta^F_{\mathrm{d}\mu}f$ and $\Delta^F f$
can be revealed by the following analog of Lemma \ref{lemma-5}.

\begin{lemma}\label{lemma-6}
Suppose $f$ is a transnormal function on a Lorentz Finsler manifold $(M,F)$ such that the integral curves of $\nabla^F f$ are $F$-unit speed geodesics, then we have
\begin{equation}\label{024}
\Delta^F f(x)=\Delta^F_{\mathrm{d} \mu} f(x)+S^F_{d\mu}(x,\nabla^F f(x)),
\end{equation}
in which the S-curvature $S^F_{d\mu}$ for the smooth measure $\mathrm{d}\mu$.
\end{lemma}

\begin{proof}
Denote $\sigma(x)$ the
positive smooth function in the local presentation
$\mathrm{d}\mu=\sigma(x)\mathrm{d}x^1\cdots \mathrm{d}x^n$ for the smooth measure $\mathrm{d}\mu$. The divergence operator defining $\Delta^F f$ is with respect to the smooth measure
\begin{eqnarray*}
\sqrt{|\det(g_{ij}(x,\nabla^F f(x)))|}\mathrm{d}x^1\cdots \mathrm{d}x^n
&=&\frac{\sqrt{|\det(g_{ij}(x,\nabla^F f(x)))|}}{\sigma(x)}\cdot\sigma(x)\mathrm{d}x^1\cdots \mathrm{d}x^n\\
&=&e^{\tau^F_{\mathrm{d}\mu}(x,\nabla^F f(x))} \mathrm{d}\mu,
\end{eqnarray*}
in which $\tau^F_{\mathrm{d}\mu}$ is the distortion function defining the S-curvature $S^F_{\mathrm{d}\mu}$ for
$F$ and $\mathrm{d}\mu$. Using the definition (\ref{022}) of $\Delta^F f$ and the Lebniz property for Lie derivative, we get
\begin{eqnarray*}
\Delta^F f&=&
\mathrm{div}_{e^{\tau^F_{\mathrm{d}\mu}(x,\nabla^F f(x))}\mathrm{d}\mu}\nabla^F f=\frac{L_{\nabla^F f}(e^{\tau^F_{\mathrm{d}\mu}(x,\nabla^F f(x))}\mathrm{d}\mu)}{e^{\tau^F_{\mathrm{d}\mu}(x,\nabla^F f(x))}\mathrm{d}\mu}\nonumber\\
&=&\frac{[(\nabla^F f)e^{\tau^F_{\mathrm{d}\mu}
(x,\nabla^F f(x))}] \mathrm{d}\mu+e^{\tau^F_{\mathrm{d}\mu}
(x,\nabla^F f(x))} L_{\nabla^F f}\mathrm{d}\mu}{e^{\tau^F_{\mathrm{d}\mu}
(x,\nabla^F f(x)))}\mathrm{d}\mu}\nonumber\\
&=&\Delta^F_{\mathrm{d} \mu} f+(\nabla^F f)\tau^F_{\mathrm{d}\mu}(x,\nabla^F f(x)).
\end{eqnarray*}

Let $\gamma(t)$ be the integral curve of $\nabla^F f$ with $\gamma(0)=x$. By our assumption, it is a $F$-unit speed geodesic.
So
\begin{eqnarray*}
(\nabla^F f)\tau^F_{\mathrm{d}\mu}(x,\nabla^F f(x))
=\frac{\mathrm{d}}{\mathrm{d}t}|_{t=0}\tau^F_{
\mathrm{d}\mu^F_{\mathrm{d}\mu}}(\gamma(t),\dot{\gamma}(t))
=S^F_{\mathrm{d}\mu}(x,\nabla^F f(x)),
\end{eqnarray*}
which ends the proof of Lemma \ref{lemma-5}.
\end{proof}


\begin{remark} We believe that similar calculation as in Finsler geometry can show
the gradient field $\nabla^F f$ for a transnormal function $f$ on a Lorentz Finsler manifold generates geodesics with possibly nonconstant speeds, i.e., Lemma \ref{lemma-6} is valid for all transnormal functions.
\end{remark}
\section{Correspondence between isoparametric hypersurface}
\subsection{Proof of Theorem \ref{main-thm}}
Let $F$ be a Finsler metric on $M$ and $V$ a homothetic vector field
for $F$ with dilation $c$ and $F(-V(x))>1$ everywhere.
Denote $\Psi_t$ the flow of local diffeomorphisms generated by $V$.
Let $\tilde{F}:\mathcal{A}\rightarrow\mathbb{R}_{>0}$ be the Lorentz Finsler metric induced by the navigation process with the datum $(F,V)$. For each $x\in M$, $\mathcal{A}_x=\mathcal{A}\cap T_xM$ is the conic convex open subset consisting of all $\tilde{y}=y+
F(x,y)V(x)$ with $\langle y,V(x)\rangle_y^F<-F(x,y)$.

Let $f$ be a normalized transnormal function locally defined around $x_0$ in the Finsler manifold
$(M,F)$ satisfying
\begin{equation}\label{028}
f(x_0)=0,\quad F(\nabla f)\equiv1,\quad \mbox{and}\quad \langle\nabla f(x),V(x)\rangle^F_{\nabla f(x)}<-1\mbox{ around }x_0.
\end{equation}

Denote $M_t=f^{-1}(t)$ the level sets of $f$ in some sufficiently small neighborhood of $x_0$, and define the smooth map $\Psi$ around $x_0$ such that $\Psi|_{M_{\alpha_c(t)}}=\Psi_t$.
\begin{lemma}\label{lemma-4}
$\Psi$ is an orientation reversing diffeomorphism around $x_0$.
\end{lemma}
\begin{proof}The smooth map $\Psi$ fixes each
$x\in M_0$. We only need to observe the tangent map $\Psi_*:T_xM\rightarrow T_xM$ has negative determinant for $x\in M_0$ sufficiently close to $x_0$. With respect to the inner product $\langle\cdot,\cdot\rangle_{\nabla^F f(x)}^F$, $T_xM$ can be orthogonally decomposed as the sum of $T_xM_0$ and $\mathbb{R}\nabla^F f(x)$.
The restriction of $\Psi_*$ to $T_xM_0$ is identity, and it maps $\nabla^F f(x)$ to $\nabla^F f(x)+V(x)$.
So $\det(\Psi_*)=\langle\nabla^F f(x),\nabla^F f(x)+V(x)\rangle^F_{\nabla^F f(x)}=1+\langle\nabla^F f(x),V(x)\rangle^F_{\nabla^F f(x)}$. When $x\in M_0$ is sufficiently close to $x_0$,
$\langle\nabla^F f(x),V(x)\rangle^F_{\nabla^F f(x)}<-F(x,\nabla^F f(x))=-1$, so
we have $\det(\Psi_*)<0$.
\end{proof}

Lemma \ref{lemma-4} enables us to define the smooth function $\tilde{f}$ around $x_0$ such that $\tilde{f}^{-1}(t)=\widetilde{M}_t=
\Psi(M_{\alpha_c(t)})$.

In the discussion below, we fix an arbitrary $t$ which is sufficiently close to 0, $x\in M_{\alpha_c(t)}$ and
$\tilde{x}=\Psi(x)=\Psi_t(x)$ which are sufficiently close to $x_0$. Denote $\gamma(\cdot)$ the integral curve of
$\nabla^F f$ with $\gamma(\alpha_c(t))=x$ and $\tilde{\gamma}(\cdot)=\Psi(\gamma(\alpha_c(\cdot)))$.
Then $\tilde{\gamma}(0)=\gamma(0)\in M_0=\widetilde{M}_0$ is sufficiently close to $x_0$, $\tilde{\gamma}(t)=\tilde{x}$, $\tilde{f}(\tilde{x})=t$, $f(x)=\alpha_c(t)$ and
$2cf(x)+1=e^{2ct}$.

Firstly, we prove the transnormal property for $\tilde{f}$.

By Lemma 4.1 in \cite{Xu2018}, $\gamma(\cdot)$ is a unit speed geodesic on $(M,F)$. Then
by Lemma \ref{lemma-2}, $\tilde{\gamma}(\cdot)$ is a unit speed geodesic on $(M,\tilde{F})$, and $T_{\tilde{\gamma}(t)}\overline{M}_t
=\Psi_*(T_{\gamma(\alpha_c(t))}M_{\alpha_c(t)})
=(\Psi_t)_*(
T_{\gamma(\alpha_c(t))}M_{\alpha_c(t)})$ is $g^{\tilde{F}}_{\dot{\tilde{\gamma}}(t)}$-orthogonal to $\dot{\tilde{\gamma}}(t)$.
So $\nabla^{\tilde{F}}\tilde{f}$ exists around $x_0$.
For arbitrary $\tilde{x}$, $\nabla^{\tilde{F}}\tilde{f}(\tilde{x})=\dot{\tilde{\gamma}}(t)$ which is a $\tilde{F}$-unit vector.
Obviously, $\tilde{f}(x_0)=f(x_0)=0$. So
we see
\begin{lemma} \label{lemma-12}
The function
$\tilde{f}$ is
a transnormal function for $\tilde{F}$ and it is normalized at $x_0$.
\end{lemma}

Secondly, we compare
$\Delta^F_{\mathrm{d} \mu^F_{\mathrm{BH}}} f(x)$ and $\Delta^{\tilde{F}}_{\mathrm{d} \mu^F_{\mathrm{BH}}}\tilde{f}(\tilde{x})=(\Psi^*\Delta^{\tilde{F}}_{\mathrm{d} \mu^F_{\mathrm{BH}}}\tilde{f})(x)$.
Since $\Delta^F_{\mathrm{d} \mu^F_{\mathrm{BH}}} f=\mathrm{div}_{\mathrm{d} \mu^F_{\mathrm{BH}}}\nabla^F f$ and
$\Delta^{\tilde{F}}_{\mathrm{d} \mu^F_{\mathrm{BH}}}\tilde{f}
=\mathrm{div}_{\mathrm{d} \mu^F_{\mathrm{BH}}}\nabla^{\tilde{F}}\tilde{f}$,
we need to compare $\nabla^F f$ and $\mathrm{d}\mu^F_{\mathrm{BH}}$ at $x$ with
$\nabla^{\tilde{F}}\tilde{f}$ and $\mathrm{d}\mu^F_{\mathrm{BH}}$ at $\tilde{x}$ respectively.

Direct calculation shows
\begin{equation}\label{019}
\nabla^{\tilde{F}}\tilde{f}(\tilde{x})=\dot{\tilde{\gamma}}(t)=
\Psi_*(\frac{\mathrm{d} }{\mathrm{d} t}\gamma(\alpha_c(t)))
=\Psi_*(e^{2ct}\dot{\gamma}(\alpha_c(t)))
=\Psi_*((2cf(x)+1)\nabla^F f(x)).
\end{equation}

Denote $\Phi=\Psi_{-t}\circ\Psi$. It is a local diffeomorphism around $x_0$ fixing each point in $M_{\alpha_c(t)}$ (so it fixes $x$). The tangent map $\Phi_*:T_xM\rightarrow T_xM$ maps $T_xM_{\alpha_c(t)}$ identically to itself. Meanwhile, $\Phi_*$ maps
$\frac{\mathrm{d} }{\mathrm{d} t}{\gamma}(\alpha_c(t))$ to
$$(\Psi_{-t})_*(\dot{\tilde{\gamma}}(t))=
(\Psi_{-t})_*((\Psi_t)_*(\frac{\mathrm{d} }{\mathrm{d} t}{\gamma}(\alpha_c(t))+
V(\tilde{x}))=\frac{\mathrm{d} }{\mathrm{d} t}{\gamma}(\alpha_c(t)))+V(x).$$
So by (\ref{008}),
\begin{eqnarray*}
\Phi^*(\mathrm{d}\mu^F_{\mathrm{BH}}(x))&=&|\det(\Phi_*|_{T_xM})|\cdot \mathrm{d}\mu^F_{\mathrm{BH}}(x)\nonumber\\
&=&|\langle\dot{\gamma}(\alpha_c(t))+e^{-2ct}V(\alpha_c(t)),
\dot{\gamma}(\alpha_c(t))
\rangle^F_{\dot{\gamma}(x)}|\cdot \mathrm{d}\mu^F_{\mathrm{BH}}\nonumber\\
&=&|1+e^{-2ct}(c_0(x)-2c\alpha_c(t))|\cdot \mathrm{d}\mu^F_{\mathrm{BH}}(x)\nonumber\\
&=&e^{-2ct}|1+c_0(x)| \mathrm{d}\mu^F_{\mathrm{BH}}(x),
\end{eqnarray*}
where $c_0(x)=\langle\dot{\gamma}(0),V(\gamma(0))
\rangle^F_{\dot{\gamma}(0)}<-1$ is a smooth function around $x_0$ which is constant along each integral curve of $\nabla^F f$ (so we have $(\nabla^F f) c_0(x)=0$).

By the homothetic property of $\Psi_t$, i.e., $\Psi_t^*(\mathrm{d}\mu^F_{\mathrm{BH}}
(\tilde{x}))=e^{-2cnt}\mathrm{d}\mu^F_{\mathrm{BH}}(x)$, we have
\begin{eqnarray}\label{020}
\Psi^*(\mathrm{d}\mu^F_{\mathrm{BH}}(\tilde{x}))&=&
\Phi^*\Psi_t^*(\mathrm{d}\mu^F_{\mathrm{BH}}(\tilde{x}))
=\Phi^*(e^{-2cnt}\mathrm{d}\mu^F_{\mathrm{BH}}(x))\nonumber\\
&=&e^{-2c(n+1)t}|1+c_0(x)|\mathrm{d}\mu^F_{\mathrm{BH}}(x)\nonumber\\
&=&|1+c_0(x)|(2cf(x)+1)^{-n-1}\mathrm{d}\mu^F_{\mathrm{BH}}(x).
\end{eqnarray}

Now we are ready to prove the following lemma comparing $\Delta^F_{\mathrm{d} \mu^F_{\mathrm{BH}}} f$
and $\Delta^{\tilde{F}}_{\mathrm{d} \mu^F_{\mathrm{BH}}}\tilde{f}$.
\begin{lemma}\label{lemma-7}
Let $f$ be a normalized transnormal function for $F$ satisfying (\ref{028}), then
\begin{equation}\label{021}
\Psi^*\Delta^{\tilde{F}}_{\mathrm{d} \mu^F_{\mathrm{BH}}}
\tilde{f}=(2cf+1)
\Delta^F_{\mathrm{d} \mu^F_{\mathrm{BH}}}f-2cn.
\end{equation}
\end{lemma}
\begin{proof} The proof uses (\ref{019}), (\ref{020}) and the same calculation as for Lemma 7.4 in \cite{XMYZ2020}, i.e.,
\begin{eqnarray*}
\Psi^*\Delta^{\tilde{F}}_{\mathrm{d} \mu^F_{\mathrm{BH}}}\tilde{f}&=&
\Psi^*\mathrm{div}_{\mathrm{d} \mu^F_{\mathrm{BH}}}(\Psi_*((2cf+1)\nabla^F f))
=\mathrm{div}_{\Psi^* \mathrm{d}\mu^F_{\mathrm{BH}}}((2cf+1)\nabla^F f)\\
&=&\mathrm{div}_{|1+c_0(x)|(2cf+1)^{-n-1}\mathrm{d}\mu^F_{\mathrm{BH}}}
((2cf+1)\nabla^F f)\\
&=&\frac{L_{(2cf+1)\nabla^F f}(|1+c_0(x)|(2cf+1)^{-n-1}\mathrm{d}\mu^F_{\mathrm{BH}})}{
|1+c_0(x)|(2cf+1)^{-n-1}\mathrm{d}\mu^F_{\mathrm{BH}}}\nonumber\\
&=&(\nabla^F f)(2cf+1)+
(2cf+1)\frac{L_{\nabla^F f}(|1+c_0(x)|(2cf+1)^{-n-1}\mathrm{d}\mu^F_{\mathrm{BH}})}{
|1+c_0(x)|(2cf+1)^{-n-1}\mathrm{d}\mu^F_{\mathrm{BH}}}\nonumber\\
&=&2c+(2cf+1)\nabla^F_{\mathrm{d} \mu^F_{\mathrm{BH}}}f+
(2cf+1)(\nabla^F f)(\ln(2cf+1)^{-n-1})\nonumber\\
& &+(2cf+1)(\nabla^F f)(\ln|1+c_0(x)|)\nonumber\\
&=&(2cf+1)\Delta^F_{\mathrm{d} \mu^F_{\mathrm{BH}}}f-2cn,
\end{eqnarray*}
in which $(\nabla^F f)(\ln|1+c_0(x)|)$ vanishes because
$c_0(x)$ is constant on each integral curve of
$\nabla^Ff$.
\end{proof}

Thirdly, we compare $\Delta^F f(x)$ and $\Delta^{\tilde{F}}\tilde{f}(\tilde{x})=(\Psi^*\Delta^{\tilde{F}}\tilde{f})(x)$. We have observed in above discussion that the integral curves of $\nabla^Ff$ and $\nabla^{\tilde{F}}\tilde{f}$ are  the unit speed geodesics  $\gamma(\cdot)$ and $\tilde{\gamma}(\cdot)$  for $F$ and $\tilde{F}$ respectively, so by Lemma \ref{lemma-5} and Lemma \ref{lemma-6}, we have
\begin{eqnarray}
\Delta^F f(x)&=&\Delta^{F}_{\mathrm{d} \mu^F_{\mathrm{BH}}} f(x)+ S^F_{\mathrm{d}\mu^F_{\mathrm{BH}}}(x,\nabla^F f(x)),
\mbox{ and}\label{025}\\
\Delta^{\tilde{F}}\tilde{f}(\tilde{x})&=&\Delta^{\tilde{F}}_{\mathrm{d} \mu^F_{\mathrm{BH}}}
\tilde{f}(\tilde{x})+
S^{\tilde{F}}_{\mathrm{d}\mu^F_{\mathrm{BH}}}(\tilde{x},\nabla^{\tilde{F}} \tilde{f}(\tilde{x})),\label{026}
\end{eqnarray}
in which the S-curvatures are for $\mathrm{d}\mu^F_{\mathrm{BH}}$.
By Lemma \ref{lemma-11} and Lemma \ref{lemma-8},
\begin{eqnarray}
S_{\mathrm{d}\mu^F_{\mathrm{BH}}}^{\tilde{F}}(\tilde{x},\nabla^{\tilde{F}}\tilde{f}(\tilde{x}))
&=&S_{\mathrm{d}\mu^F_{\mathrm{BH}}}^{\tilde{F}}(\Psi_t(\tilde{\gamma}(\alpha_c(t))),
(\Psi_t)_*(\frac{\mathrm{d} }{\mathrm{d} t}\gamma(\alpha_c(t)))+V(\tilde{\gamma}(t)))
\nonumber\\&=&S_{\mathrm{d}\mu^F_{\mathrm{BH}}}^F(\Psi_t(\tilde{\gamma}(\alpha_c(t))),
(\Psi_t)_*(\frac{\mathrm{d} }{\mathrm{d} t}\gamma(\alpha_c(t))))+(n+1)c\nonumber\\
&=&S_{\mathrm{d}\mu^F_{\mathrm{BH}}}^F(\tilde{\gamma}(\alpha_c(t)),\frac{\mathrm{d} }{\mathrm{d} t}\gamma(\alpha_c(t)))
+(n+1)c\nonumber\\&=&e^{2ct}S_{\mathrm{d}\mu^F_{\mathrm{BH}}}^F(x,\nabla^F f(x))+(n+1)c\nonumber\\
&=&(2cf(x)+1)S_{\mathrm{d}\mu^F_{\mathrm{BH}}}^F(x,\nabla^Ff(x))+(n+1)c.\label{027}
\end{eqnarray}
By Lemma \ref{lemma-7} and (\ref{025})-(\ref{027}),
\begin{eqnarray*}
\Delta^{\tilde{F}}\tilde{f}(\tilde{x})&=&\Delta^{\tilde{F}}_{\mathrm{d} \mu^F_{\mathrm{BH}}}\tilde{f}(\tilde{x})
+S_{\mathrm{d}\mu^F_{\mathrm{BH}}}^{\tilde{F}}(\tilde{x},\nabla^{\tilde{F}}\tilde{f}(\tilde{x}))\\
&=&(2cf(x)+1)\Delta^F_{\mathrm{d} \mu^F_{\mathrm{BH}}}f(x)-2nc+(2cf(x)+1)S_{\mathrm{d}\mu^F_{\mathrm{BH}}}^{F}(x,\nabla^F f(x))+(n+1)c\\
&=&(2cf(x)+1)\Delta^F f(x)-(n-1)c.
\end{eqnarray*}
So we have the following lemma
\begin{lemma}\label{lemma-9}
Let $f$ be a normalized transnormal function for $F$ satisfying (\ref{028}), then
\begin{equation}\label{021}
\Psi^*\Delta^{\tilde{F}}
\tilde{f}=(2cf+1)
\Delta^F f-(n-1)c.
\end{equation}
\end{lemma}

Finally, we summarize above discussion to the proof of Theorem \ref{main-thm}.

\begin{proof}[Proof of Theorem \ref{main-thm}]
Suppose that $f$ is a $\mathrm{d}\mu^F_{\mathrm{BH}}$-isoparametric function for $F$ satisfying (\ref{028}). We see from Lemma \ref{lemma-4} and Lemma \ref{lemma-12} that the function $\tilde{f}$ determined by $\tilde{f}^{-1}(t)=\widetilde{M}_t=\Psi_{t}(M_{\alpha_c(t)})=
\Psi_t(f^{-1}(\alpha_c(t)))$ around $x_0$ is a transnormal function normalized at $x_0$. Further more, since $\Delta^F_{\mathrm{d}\mu^F_{\mathrm{BH}}}f$ is constant on each $M_{\alpha_c(t)}$,
Lemma \ref{lemma-7} indicates that $\Delta^{\tilde{F}}_{\mathrm{d} \mu^F_{\mathrm{BH}}}\tilde{f}$ is constant on each $\widetilde{M}_t$ as well, i.e., $\tilde{f}$ is $\mathrm{d}\mu^F_{\mathrm{BH}}$-isoparametric for $\tilde{F}$. In particular, $\tilde{f}^{-1}(0)=\widetilde{M}_0=\widetilde{M}_0=f^{-1}(0)$
is a $\mathrm{d}\mu^F_{\mathrm{BH}}$-isoparametric hypersurface for the Lorentz Finsler metric $\tilde{F}$.

When $f$ is an isoparametric function for $F$ satisfying (\ref{028}), we use Lemma \ref{lemma-9} instead of Lemma \ref{lemma-7}. The other argument is similar.
\end{proof}
\subsection{Some explicit examples}

Theorem \ref{main-thm} can provide many examples of locally defined isoparametric hypersurfaces for
a Lorentz Finsler metric $\tilde{F}$ induced by homothetic navigation. Here we discuss some examples when $\tilde{F}$ is
a Funk metric of Lorentz Randers type.


Denote $\Omega=\{x\in\mathbb{R}^n\mbox{ with }|x|\leq1\}$,
in which $|x|=\langle x,x\rangle^{1/2}=\sum_{i=1}^n (x^i)^2$ for
$x=(x^1,\cdots,x^n)$ is
the standard Euclidean norm on $\mathbb{R}^n$. Then a Lorentz Finsler metric $\tilde{F}$ on $M=\mathbb{R}^n\backslash\Omega$
can be induced by the navigation process with the datum $(F,V)$, where $F$ is the standard Euclidean metric on $\mathbb{R}^n$ and $V(x)=-x$ is a homothetic field with dilation $\tfrac12$. The metric
$\tilde{F}$ generalizes the Funk metric in Finsler geometry, so we may call it a Lorentz Funk metric. It is of Randers type because it can be explicitly presented as
\begin{eqnarray*}
\tilde{F}(x,y)=\frac{\sqrt{\langle x,y\rangle^2-(|x|^2-1)|y|^2}-\langle x,y\rangle}{|x|^2-1},
\end{eqnarray*}
where $x\in M=\mathbb{R}^n\backslash{\Omega}$, i.e., $|x|>1$ and $y\in \mathcal{A}_x=
\{y|\langle x,y\rangle^2-(|x|^2-1)|y|^2>0\}\subset
\mathbb{R}^n=T_xM$. Since $F$ is Riemannian, the $\mathrm{d}
\mu^F_{\mathrm{BH}}$-isoparametric
properties and isoparametric properties for $\tilde{F}$ coincide.

Let $x_0$ by any point in $M=\mathbb{R}^n\backslash\Omega$ and $N$ an isoparametric hypersurface in $(\mathbb{R}^n,F)$ passing $x_0$, i.e., $N$ is
a hyperplane, a sphere, or a spherical cylinder \cite{Se1926}. Denote $\mathbf{n}$ the $F$-unit normal vector of $N$ at $x_0$. Then Theorem \ref{main-thm} implies that $N$ is isoparametric for $\tilde{F}$ in a sufficiently small neighborhood of $x_0$ when $|\langle x,\mathbf{n}\rangle|>1$. Notice that
there are two isoparametric functions $f_{\pm}$ for $F$ normalized at $x$ with $f^{-1}_{\pm}(0)=N$ and $\nabla^F f_{\pm}(x_0)=\pm\mathbf{n}$. To get the isoparametric function $\tilde{f}$ for $\tilde{F}$ in Theorem \ref{main-thm}, we need to take $f$ from $f_{\pm}$ such that $\langle x_0,\nabla^F f(x_0)\rangle>1$.

Some isoparametric hypersurfaces in $\tilde{F}$ exist globally and the corresponding normalized isoparametric functions can be easily presented.
Fix an arbitrary $a>1$. Then the function $f_1(x)=|x|-a$  is an isoparametric function for $F$ normalized at any point on the sphere $S_o(a)=\{x| |x|=a\}$. By Theorem \ref{main-thm}, $\tilde{f}_1(x)=\ln\frac{a-1}{|x|-1}$
is an isoparametric
function for $\tilde{F}$ in some neighborhood of  $S_o(a)$. So each sphere $S_o(a)$ with
$a>1$ is an isoparametric hypersurface for $\tilde{F}$.
The function $f_2(x)=x^1-a$ for $x=(x^1,\cdots,x^n)$ is another isoparametric function for $F$ normalized at any point on the hyperplane $x^1=a$.
By Theorem \ref{main-thm}, $\tilde{f}_2(x)
=\ln\tfrac{a-1}{x^1-1}$ is an isoparametric function for $\tilde{F}$ in some neighborhood of this hyperplane. So each hyperplane $x^1=a$ with $a>1$ is an isoparametric hypersurface for $\tilde{F}$.
%
\bigskip

\noindent
{\bf Acknowledgement}.
This paper is supported by Beijing Natural Science Foundation
(Z180004), National Natural Science
Foundation of China (11771331, 11821101, 12001007),
Capacity Building for Sci-Tech  Innovation -- Fundamental Scientific Research Funds (KM201910028021), Natural Science Foundation of Anhui province (2008085QA03, 1908085QA03). The first author thank Anhui University and Sichuan University for hospitality during the preparation of this paper. The authors also thanks Qun He and Miguel Angel Javaloyes for helpful discussions and suggestions.

\end{document}